\documentclass[a4paper,12pt]{article}
\usepackage{amsthm,amssymb,amsmath,amsfonts,latexsym}

\theoremstyle{theorem}
\newtheorem{theorem}{Theorem}
\newtheorem{lemma}{Lemma}
\theoremstyle{definition}
\newtheorem{definition}{Definition}
\newtheorem{remark}{Remark}

\usepackage[unicode=true,hyperfootnotes=true,pdftitle={Incompleteness},pdfsubject={Mathematical Logic},pdfkeywords={Logic},
pdfauthor=CCPS,pdfstartview=FitH,pdfstartpage=1,bookmarks=true,bookmarksopen=true,bookmarksopenlevel=0,
linktocpage=false,pdfdisplaydoctitle,pdfpagelabels=true,breaklinks=true,pdfpagemode=UseOutlines,hyperfigures=true,nesting=true,pageanchor=true,
pdfnewwindow=true,pdfhighlight=/I,breaklinks=true,colorlinks=true,citecolor=blue,filecolor=blue,linkcolor=blue,urlcolor=blue]{hyperref}
\usepackage{setspace}
\usepackage{indentfirst}
\usepackage{enumerate}                  
\usepackage{geometry}
\usepackage[bottom,perpage]{footmisc}           

\begin{document}
\title{\bf G\"odel's second incompleteness theorem for $\Sigma_n$-definable theories}
 \author{Conden Chao\footnote{Department of Philosophy, Peking University. My correspondence address is: Room 223 of Building 4th, Changchun New Garden, Peking University, P.~C.~100872, Peking, CHINA. Email: {\tt 476250089@qq.com}.}\\Payam Seraji\footnote{
 Faculty of Mathematical Sciences, University of Tabriz, 29 Bahman Boulevard, P.~O.~Box~5166617766, Tabriz, IRAN. Email: {\tt p$\_$seraji@tabrizu.ac.ir}.}}




\date{\today}



\maketitle
\begin{abstract}
 G\"odel's second incompleteness theorem is generalized by showing
that if the set of axioms of a theory $T\supseteq \textsf{PA}$ is $\Sigma_{n+1}$-definable and $T$ is $\Sigma_n$-sound, then $T$ dose not prove the sentence
$\Sigma_n\text{-}\textsf{Sound}(T)$ that expresses the $\Sigma_n$-soundness of $T$. The optimality of the generalization is shown
by presenting a  $\Sigma_{n+1}$-definable (indeed a complete
$\Delta_{n+1}$-definable) and $\Sigma_{n-1}$-sound theory $T$ such that $\textsf{PA} \subseteq T$ and $\Sigma_{n-1}\text{-}\textsf{Sound}(T)$ is provable in $T$. It is also proved that no recursively enumerable and $\Sigma_1$-sound theory of arithmetic, even very weak theories which do not contain  Robinson's Arithmetic, can prove its own $\Sigma_1\text{-}{\rm soundness}$.
\par\vbox{} Keywords: G\"odel's second incompleteness, $\Sigma_n$-definable, $\Sigma_n$-sound, $\Sigma_n\textsf{-Sound}$, strong provability predicate
\end{abstract}

\section{Introduction}
G\"odel's second  incompleteness theorem states that for any recursively enumerable and sufficiently strong  (say any extension of Peano's Arithmetic $\textsf{PA}$) theory $T$, $T\not\vdash\textsf{Con}(T)$, where $\textsf{Con}(T)$ is the arithmetical sentence expressing the the consistency of $T$ (see e.g. \cite{Tpol2003,Mfoa1998,Acml2006,Aigt2013}). This consistency statement is usually built from a ``provability predicate'' such as $\textsf{Con}(T)=_{df}\neg\textsf{Pr}_T(\ulcorner\bot\urcorner)$, where $\textsf{Pr}_T$ satisfies the 
derivability conditions: 
\begin{enumerate}
\item[$D_1$:] if $T\vdash\phi$ then $T\vdash \textsf{Pr}_T(\ulcorner \phi \urcorner)$,
\item[$D_2$:] $T\vdash\textsf{Pr}_T(\ulcorner\phi \rightarrow \psi\urcorner)\rightarrow[\textsf{Pr}_T(\ulcorner \phi \urcorner) \rightarrow\textsf{Pr}_T(\ulcorner \psi \urcorner)]$, and
\item[$D_3$:] $T\vdash\textsf{Pr}_T(\ulcorner \phi \urcorner) \rightarrow\textsf{Pr}_T(\ulcorner\textsf{Pr}_T(\ulcorner \phi \urcorner) \urcorner)$.
\end{enumerate}
A natural question that comes to mind is that what happens to G\"odel's second  incompleteness theorem for non-recursively enumerable theories? For considering the phenomenon of G\"odel's second  incompleteness theorem for general (not necessarily recursively enumerable) theories $T$, we should be able to write down $\textsf{Con}(T)$ or equivalently $\textsf{Pr}_T$ (noting that $\textsf{PA}\vdash\textsf{Pr}_T(\ulcorner\phi\urcorner)\leftrightarrow\neg \textsf{Con}(T+\neg\phi)$ for any formula $\phi$); thus we can only consider  definable theories. Let us note that if the provability predicate of a definable theory $T$ satisfies the derivability conditions then it can be shown that $T$ cannot prove its consistency by the usual argument.
\par But if $\textsf{Pr}_T$ does not satisfy the derivability conditions, then G\"odel's second  incompleteness theorem may not hold anymore; see \cite[pp. 263--264]{Eldt1975} for an example of a consistent $\Delta_2$-definable extension of  $\textsf{PA}$ which proves its own (standard) consistency statement (Section~5 of \cite{Elmc2012} contains a more modern treatment).
So, the derivability conditions may not hold for definable theories in general, even if they are sufficiently strong, e.g. contain $\textsf{PA}$, for the reason that $\textsf{Pr}_T$ is not then necessarily a $\Sigma_1$-formula,  so $D_1$ or $D_3$ above may not hold anymore.
One of the earliest instances of G\"odel's second  incompleteness theorem for non-recursively enumerable (but definable) theories is Jeroslow's Theorem~6 in \cite[p.~264]{Eldt1975} stating (in an equivalent rewording) that $\Delta_2$-definable extensions of $\textsf{PA}$  cannot prove their own  $\Sigma_2$-soundness, provided that they are $\Sigma_1$-sound (cf. \cite{Nsgt2011} for the equivalence of definitions) and satisfy some further technical conditions.
\par A theory is called {\em $\Sigma_n$-sound} if it cannot prove a false $\Sigma_n$-sentence.
For any theory $T$, and any $n\in\mathbb{N}$, the $\Sigma_n$-soundness of $T$ is equivalent to its consistency with $\Pi_n\text{-}{\rm Th}(\mathbb{N})$, the set of all true $\Pi_n$-sentences (see \cite{Mopa1991} for the notation). For simplicity we will write $\Sigma_n\text{-}\textsf{Sound}(T)$ instead of $\textsf{Con}(T\cup\Pi_n\text{-}{\rm Th}(\mathbb{N}))$.
Here we will show that G\"odel's second  incompleteness theorem holds for $\Sigma_{n+1}$-definable and $\Sigma_n$-sound theories, in the sense that if $T$ is a $\Sigma_{n+1}$-definable and $\Sigma_n$-sound theory containing $\textsf{PA}$, then $T$ dose not prove $\Sigma_n\text{-}\textsf{Sound}(T)$ (Theorem~\ref{CS009} below). This result is a bit stronger than a version which follows quickly from the well-known facts about, the so called, \textit{strong provabilty predicates} (Theorem~\ref{CS007} below).
 We will also show the optimality of this result by presenting  a $\Sigma_{n+1}$-definable (indeed complete  $\Delta_{n+1}$-definable) $\Sigma_{n-1}$-sound extension of $\textsf{PA}$ which proves its own $\Sigma_{n-1}$-soundness.


\section{Generalized G\"odel's second incompleteness theorem}
A theory $T$ is {\em definable} when there exists a formula $\textsf{Axiom}_T(x)$ such that for every natural number $n$, $\textsf{Axiom}_T(n)$ holds just in case $n$ is the G\"odel number of an axiom of $T$. The formula $\textsf{ConjAx}_T(x)$ indicates that $x$ is the G\"odel number of a formula which is the conjunction of some axioms of $T$. Let $\textsf{Proof}(y,x)$ be the proof relation in first-order logic, saying that $y$ is the G\"odel code of a proof of a formula with G\"odel number $x$. Thus, the consistency of a definable theory $T$,  i.e., $\textsf{Con}(T)$, can be written as $$\forall s,u\big[\textsf{ConjAx}_T(s)\rightarrow\neg\textsf{Proof}(u,\ulcorner s\rightarrow\bot\urcorner)\big].$$ So, we can write $\Sigma_n\text{-}\textsf{Sound}(T)=_{df}\textsf{Con}(T\cup\Pi_n\text{-}{\rm Th}(\mathbb{N}))$ as $$\forall s,t,u\big[\textsf{ConjAx}_T(s)\wedge\Pi_n\text{-}\textsf{True}(t)\rightarrow\neg\textsf{Proof}(u,\ulcorner s\wedge t\rightarrow\bot\urcorner)\big],$$
where the formula $\Pi_n\text{-}\textsf{True}(x)$ defines the set $\Pi_n\text{-}{\rm Th}(\mathbb{N})$. We call a theory $T$ an \textit{extension} of $\textsf{PA}$ if  $\mathbb{N}\vDash \forall x \big[\textsf{Axiom}_{\sf PA}(x)\rightarrow \textsf{Axiom}_T(x)\big]$, where $\textsf{Axiom}_{\sf PA}$ is a $\Delta_0$-formula defining the set of axioms of ${\sf PA}$. $T$ is an  \textit{explicit} (or \textit{provable}) extension of ${\sf PA}$ when we have ${\sf PA}\vdash \forall x \big[\textsf{Axiom}_{\sf PA}(x)\rightarrow \textsf{Axiom}_T(x)\big]$.
\par For each $n\in \mathbb{N}$, let ${\tt Pr}^{(n+1)}(x)$ be the provability predicate of theory $T=\textsf{PA}+\Pi_n\text{-}{\rm Th}(\mathbb{N})$. The predicate ${\tt Pr}^{(n+1)}(x)$ is an example of a \textit{strong provability predicate of degree $n+1$} (cf. Definition~2.1 of \cite{Ospa1993})  which means it satisfies the following conditions:
\begin{enumerate}
\item[$C_1$:] ${\tt Pr}^{(n+1)}(x)\in\Sigma_{n+1}$;
\item[$C_2$:] ${\sf PA}\vdash{\tt Pr}^{(n+1)}(\ulcorner \phi\rightarrow \psi\urcorner)\rightarrow\big[{\tt Pr}^{(n+1)}(\ulcorner \phi\urcorner)\rightarrow{\tt Pr}^{(n+1)}(\ulcorner \psi\urcorner)\big]$ for every $\phi,\psi$;
\item[$C_3$:] ${\sf PA}\vdash \phi\rightarrow{\tt Pr}^{(n+1)}(\ulcorner \phi\urcorner)$ for every $\phi\in\Sigma_{n+1}$;
\item[$C_4$:] if $\mathbb{N}\models{\tt Pr}^{(n+1)}(\ulcorner \phi\urcorner)$ then $\mathbb{N}\models \phi$ for every $\phi\in\Sigma_{n+1}$;
\item[$C_5$:] if ${\sf PA}\vdash \phi$ then ${\sf PA}\vdash{\tt Pr}^{(n+1)}(\ulcorner \phi\urcorner)$ for every $\phi$.
\end{enumerate}
 Using these properties, it can be proved that ${\tt Pr}^{(n+1)}(x)$ satisfies the L\"ob axiom (Theorem~2.2 of \cite{Ospa1993}), which is $\textsf{PA}\vdash {\tt Pr}^{(n+1)}(\ulcorner {\tt Pr}^{(n+1)}(\ulcorner \phi\urcorner)\rightarrow \phi\urcorner)\rightarrow {\tt Pr}^{(n+1)}(\ulcorner\phi\urcorner)$ (for any sentence $\phi$), and it immediately implies that

\begin{theorem}\label{CS001}
  For each $n\in \mathbb{N}$, the theory $T=\textsf{PA}+\Pi_n\text{-}{\rm Th}(\mathbb{N})$ can not prove its own consistency.
\end{theorem}
\begin{proof}
  Let $\phi=\bot$ in the L\"ob's axiom.
\end{proof}
This result can be a bit generalized by the following observation. Let $T\supseteq \textsf{PA}$ to be a $\Sigma_{n+1}$-definable theory which does not necessarily contain all $\Pi_n\text{-}{\rm Th}(\mathbb{N})$, but it is $\Sigma_n$-sound and $T$ is also an explicit extension of $\textsf{PA}$. The $\Sigma_n$-soundness of $T$ implies that the theory $T^*=T+\Pi_n\text{-}{\rm Th}(\mathbb{N})$ is consistent. Let ${\tt Pr}_{T^*}$ to be the provability predicate of $T^*$. It can be easily checked that ${\tt Pr}_{T^*}$ satisfies the properties $C_1$, $C_2$ and $C_5$ for a strong provability predicate of degree $n+1$. By Proposition~2.11 of~\cite{Rppa2005}  for every $\sigma\in\Sigma_{n+1}$ we have
$${\sf PA}\vdash\sigma\rightarrow\exists s,t,u\big[\textsf{ConjAx}_{\sf PA}(s)\wedge\Pi_n\text{-}\textsf{True}(t)\wedge\textsf{Proof}(u,\ulcorner s\wedge t\rightarrow\sigma\urcorner)\big].$$
Thus
\[{\sf PA}\vdash\sigma\rightarrow\exists s,t,u\big[\textsf{ConjAx}_{T}(s)\wedge\Pi_n\text{-}\textsf{True}(t)\wedge\textsf{Proof}(u,\ulcorner s\wedge t\rightarrow\sigma\urcorner)\big]\]
 (because ${\sf PA}\vdash\forall x\big[\textsf{Axiom}_{\sf PA}(x)\rightarrow\textsf{Axiom}_{T}(x)\big]$). Hence
${\sf PA}\vdash\sigma\rightarrow {\tt Pr}_{T^*}(\ulcorner \sigma \urcorner)$  for any $\sigma\in\Sigma_{n+1}$, so ${\tt Pr}_{T^*}$ also satisfies the property $C_3$ for a strong provability predicate of degree $n+1$. A close inspection of Theorem~2.2 of  \cite{Ospa1993} (which proves the L\"ob's axiom) reveals that the property $C_4$ is not used in its proof, so  the predicate ${\tt Pr}_{T^*}$ also satisfies the L\"ob's axiom which is the formalized G\"odel's second  incompleteness theorem. So $T^*\nvdash \textsf{Con}(T^*)$. But $\textsf{Con}(T^*)$ is exactly $\Sigma_n\text{-}\textsf{Sound}(T)$, hence
 $\Sigma_n\text{-}\textsf{Sound}(T)$ is not provable in $T^*$ and then in $T$. So we have proved the following result which is a generalization of G\"odel's second  incompleteness theorem for definable theories, noting that for extensions of ${\sf PA}$,  $\Sigma_0$-soundness is equivalent to consistency (Theorem~5 of \cite{Nsgt2011}); thus G\"odel's second theorem is the following theorem for $n=0$.
\begin{theorem}
  \label{CS007}
For any $\Sigma_{n+1}$-definable and $\Sigma_n$-sound theory $T$ which is an explicit extension of $\textsf{PA}$, i.e. $\textsf{PA}\vdash \forall x [\textsf{Axiom}_{\textsf{PA}}(x)\rightarrow \textsf{Axiom}_{T}(x)]$, we have $T\not\vdash\Sigma_n\text{-}\textsf{Sound}(T)$.
\end{theorem}
We will show that this result holds even if $\forall x [\textsf{Axiom}_{\textsf{PA}}(x)\rightarrow \textsf{Axiom}_{T}(x)]$ is not necessarily provable in $\textsf{PA}$ (Theorem \ref{CS009} below). At first we need a few lemmas. The first one is a generalization of   Craig's trick.
\begin{lemma}
\label{CS006}
  For any $n\in \mathbb{N}$, if a theory $T$ is definable by a $\Sigma_{n+1}$ formula, then it is also definable by a $\Pi_n$ formula.
\end{lemma}
\begin{proof}
   Let the $\Sigma_{n+1}$ formula ${\textsf{Axiom}_T}(x)=\exists x_1\cdots\exists x_m\psi(x,x_1,\cdots,x_m)$ define the set of axioms of $T$ (with $\psi\in\Pi_n$). This formula is logically equivalent to the formula  $\exists y\delta(x,y)=\exists y\exists x_1\!\leq\! y\cdots\exists x_m\!\leq\!y\psi(x,x_1,\cdots,x_m)$. Note that $\delta(x,y)\in\Pi_n$. So the set of sentences $\Omega=\{\phi\wedge(\overline{k}=\overline{k})\mid\mathbb{N}\vDash\delta(\ulcorner\phi
   \urcorner,{k})\}$ also axiomatizes $T$. Clearly the $\Pi_n$ formula $\textsf{Axiom}_{T'}(x)=_{df}\exists y\leq x\exists z\leq x[\delta(y,z)\wedge(x=\ulcorner y\wedge(\overline{z}=\overline{z})\urcorner)]$ defines $\Omega$.
\end{proof}

Let $\Sigma_n\text{-}\textsf{Sound}(T')$ be the sentence asserting the $\Sigma_n$-soundness of the theory $T'$ which is defined by the formula $\textsf{Axiom}_{T'}(x)$ as  above, i.e.   

$\Sigma_n\text{-}\textsf{Sound}(T')=\forall s,t,u\big[\textsf{ConjAx}_{T'}(s)\wedge\Pi_n\text{-}\textsf{True}(t)\rightarrow\neg\textsf{Proof}(u,\ulcorner s\wedge t\rightarrow\bot\urcorner)\big]$.

\begin{lemma}
\label{CS007}
  $\textsf{PA}\vdash \Sigma_n\text{-}\textsf{Sound}(T)\leftrightarrow \Sigma_n\text{-}\textsf{Sound}(T')$
\end{lemma}
\begin{proof}
  (Working in $\textsf{PA}$) For any formula $\phi$, $\textsf{Axiom}_T(\ulcorner \phi \urcorner)$ if and only if $\textsf{Axiom}_{T'}(\ulcorner \phi \wedge z=z \urcorner)$ for some suitable $z$. Obviously the set of logical consequences of  $A=\{ \phi\mid\textsf{Axiom}_T(\ulcorner \phi \urcorner)\}$ and logical consequences of $\Omega=\{ \phi~|~ \textsf{Axiom}'_T(\ulcorner \phi \urcorner)\}$ are the same. Hence they prove same sentences of the form $\chi\rightarrow \bot$ where $\chi$ is a (conjunction of) true $\Pi_n$ sentences. Therefore,  $A+\Pi_n\textsf{-True}(\mathbb{N})$ is consistent if and only if $\Omega+\Pi_n\textsf{-True}(\mathbb{N})$ is consistent.
\end{proof}
\begin{lemma}\label{CS005}
 ${\sf PA}+\Sigma_{k}\text{-}\textsf{Sound}(T)\vdash\Sigma_{k}\text{-}\textsf{Sound}(T+\phi)\vee\Sigma_{k}\text{-}\textsf{Sound}(T+\neg\phi)$ holds for any formula $\phi$ and any $k\in\mathbb{N}$ and any definable theory $T$.
\end{lemma}
\begin{proof}
Reason inside ${\sf PA}+\Sigma_{k}\text{-}\textsf{Sound}(T)$: if (on the contrary we have)
\newline\centerline{$\neg\Sigma_{k}\text{-}\textsf{Sound}(T+\phi)$ and $\neg\Sigma_{k}\text{-}\textsf{Sound}(T+\neg\phi)$} then there are $s',t',u',s'',t'',u''\in \mathbb{N}$ such that
\newline\centerline{
 $\textsf{ConjAx}_T(s')\wedge\Pi_k\text{-}\textsf{True}(t')\wedge\textsf{Proof}(u',\ulcorner s'\wedge t'\rightarrow\phi\urcorner)$ and}
\newline\centerline{
 $\textsf{ConjAx}_T(s'')\wedge\Pi_k\text{-}\textsf{True}(t'')\wedge\textsf{Proof}(u'',\ulcorner s''\wedge t''\rightarrow\neg\phi\urcorner)$.}
\noindent Then for $s=s'\wedge s'', t=t'\wedge t''$ and a suitable  $u$ we have
\newline\centerline{
$\textsf{ConjAx}_T(s)\wedge\Pi_k\text{-}\textsf{True}(t)\wedge\textsf{Proof}(u,\ulcorner s\wedge t\rightarrow\bot\urcorner)$,}
    \noindent which implies $\neg\Sigma_{k}\text{-}\textsf{Sound}(T)$, contradiction.
\end{proof}

\begin{theorem}
\label{CS008}
For any $\Pi_{n}$-definable and $\Sigma_n$-sound theory $T$ extending $\textsf{PA}$, we have that  $T\not\vdash\Sigma_n\text{-}\textsf{Sound}(T)$.
\end{theorem}
\begin{proof}
Let $T^\ast=T\cup\Pi_n\text{-}{\rm Th}(\mathbb{N})$ which is a consistent theory by the assumption of $\Sigma_n$-soundness of $T$. By the diagonal lemma there exists a sentence $\gamma$ such that ${\sf PA}\vdash\gamma\leftrightarrow\Sigma_n\text{-}\textsf{Sound}(T+\neg\gamma)$.

Firstly, we show  $T\not\vdash\gamma$ even more $T^\ast\not\vdash\gamma$: since otherwise (if $T^\ast\vdash\gamma$) there would exists some $s,t,u\in\mathbb{N}$ such that $\textsf{ConjAx}_{T}(s)\wedge\Pi_n\text{-}\textsf{True}({t})
\wedge\textsf{Proof}(\ulcorner{u},{s}\wedge {t}\rightarrow\gamma\urcorner)$ is a true ($\Pi_{n}$-)sentence. Since all true $\Pi_{n}$-sentences are provable in $\Pi_n\text{-}{\rm Th}(\mathbb{N})$ (and so in $T^\ast$) then we would have   $T^\ast\vdash\neg\Sigma_n\text{-}\textsf{Sound}(T+\neg\gamma)$ thus $T^\ast\vdash\neg\gamma$, contradiction.

Secondly, we prove $T^*\vdash\Sigma_n\text{-}\textsf{Sound}(T+\gamma)\rightarrow\gamma$: note that by Proposition~2.11 of~\cite{Rppa2005} for every $\sigma\in\Sigma_{n+1}$ we have
\newline\centerline{${\sf PA}\vdash\sigma\rightarrow\exists s,t,u\big[\textsf{ConjAx}_{\sf PA}(s)\wedge\Pi_n\text{-}\textsf{True}(t)\wedge\textsf{Proof}(u,\ulcorner s\wedge t\rightarrow\sigma\urcorner)\big].$}
Thus
\begin{equation}\label{CS002}
   T^* \vdash\sigma\rightarrow\exists s,t,u\big[\textsf{ConjAx}_{\sf PA}(s)\wedge\Pi_n\text{-}\textsf{True}(t)\wedge\textsf{Proof}(u,\ulcorner s\wedge t\rightarrow\sigma\urcorner)\big].
 \end{equation}
Since $\forall x (\textsf{Axiom}_{\textsf{PA}}(x)\rightarrow \textsf{Axiom}_{T}(x))$ is a true $\Pi_n$ sentence and $\Pi_n\text{-}{\rm Th}(\mathbb{N})\subseteq T^*$,
\begin{equation}\label{CS003}
  T^*\vdash \forall x (\textsf{Axiom}_{\textsf{PA}}(x)\rightarrow \textsf{Axiom}_{T}(x))
\end{equation}
\eqref{CS002} together with \eqref{CS003} implies that
\begin{equation}\label{CS004}
  T^* \vdash\sigma\rightarrow\exists s,t,u\big[\textsf{ConjAx}_{T}(s)\wedge\Pi_n\text{-}\textsf{True}(t)\wedge\textsf{Proof}(u,\ulcorner s\wedge t\rightarrow\sigma\urcorner)\big].
 \end{equation}
So
$T^* \vdash\sigma\rightarrow\neg\Sigma_n\text{-}\textsf{Sound}
(T+\neg\sigma)$ for any $\sigma\in\Sigma_{n+1}$. It suffices now to note that $\neg\gamma\in\Sigma_{n+1}$ thus $T^* \vdash\neg\gamma
\rightarrow\neg\Sigma_n\text{-}\textsf{Sound}(T+\gamma)$, hence $T^* \vdash \Sigma_n\text{-}\textsf{Sound}(T+\gamma) \rightarrow \gamma$.
\par Thirdly, we show  $T^*\vdash\Sigma_n\text{-}\textsf{Sound}(T)\rightarrow\gamma$.
By Lemma~\ref{CS005} we already have
\[T^*+\Sigma_{n}\text{-}\textsf{Sound}(T)\vdash\Sigma_{n}\text{-}\textsf{Sound}(T+\gamma)\vee\Sigma_{n}\text{-}\textsf{Sound}(T+\neg\gamma),\]
and so by the definition of $\gamma$  $(T\vdash\Sigma_{n}\text{-}\textsf{Sound}(T+\neg\gamma)\rightarrow\gamma)$ and the second point above $(T^*\vdash\Sigma_{n}\text{-}\textsf{Sound}(T+\gamma)\rightarrow\gamma)$ we can conclude that $T^*\vdash\Sigma_{n}\text{-}\textsf{Sound}(T)\rightarrow\gamma$.

Finally, if $T\vdash\Sigma_n\text{-}\textsf{Sound}(T)$ then by the third point above $T^*\vdash\gamma$ contradicting the first point above.
\end{proof}

\begin{theorem}
  \label{CS009}
For any $\Sigma_{n+1}$-definable and $\Sigma_n$-sound theory $T$ extending $\textsf{PA}$, we have $T\not\vdash\Sigma_n\text{-}\textsf{Sound}(T)$.
\end{theorem}
\begin{proof}
  Let ${\textsf{Axiom}_{T'}}(x)$ to be the $\Pi_n$ formula defining   $T'$ (constructed in Lemma~\ref{CS006}) which is equivalent to the theory $T$. By Lemma~\ref{CS007}, $\textsf{PA}\vdash\Sigma_n\text{-}\textsf{Sound}(T)\leftrightarrow \Sigma_n\text{-}\textsf{Sound}(T')$ and by the previous theorem $T'\nvdash \Sigma_n\text{-}\textsf{Sound}(T')$, therefore $T\nvdash \Sigma_n\text{-}\textsf{Sound}(T)$.
\end{proof}

\begin{remark}
In the above arguments ${\sf PA}$ can be replaced, everywhere, either with  the theory ${\sf I\Delta_0+Exp}$ (the fragment of ${\sf PA}$ in which the induction scheme is restricted to $\Delta_0$-formulas plus the axiom of totality of the exponential function, see e.g. \cite{Mfoa1998}), or with the theory ${\sf EA}$ (the elementary arithmetic, see e.g. \cite{Rppa2005}), since it is well-known that ${\sf I\Delta_0+Exp}$ and ${\sf EA}$ are definitionally equivalent.
\end{remark}

For any recursively enumerable and $\Sigma_1$-sound theory $T$ in the language of arithmetic
augmented with a symbol for exponential function (even very weak theories that dose not contain the Robinson's arithmetic), we can prove the following theorem as a corollary.

\begin{theorem}
If the function symbol $\textsf{exp}$ (with its standard
interpretation) is in the language of a recursively enumerable and $\Sigma_1$-sound
theory $T$, then $T\not\vdash\Sigma_1\text{-}\textsf{Sound}(T)$.
\end{theorem}
\begin{proof}
Let $T^\ast\!=\!T\cup\Pi_1\text{-}{\rm Th}(\mathbb{N})$. By the
$\Sigma_1$-soundness of $T$ the theory $T^\ast$ is consistent and
contains ${\sf EA}$ (or equivalently ${\sf I\Delta_0+Exp}$ since
both ${\sf EA}$ and ${\sf I\Delta_0+Exp}$ are $\Pi_1$-axiomatizable
in the presence of ${\sf exp}$). So, Proposition~2.11
of~\cite{Rppa2005} implies that {$ T^\ast
\vdash\sigma\rightarrow\exists s,t,u\big[\textsf{ConjAx}_{T}(s)\wedge\Pi_1\text{-}\textsf{True}(t)\wedge\textsf{Proof}(u,\ulcorner s\wedge t\rightarrow\sigma\urcorner)\big]$} (for
any $\Sigma_2$ sentence $\sigma$). Thus, by an argument similar to
the previous theorem,
$T^\ast\not\vdash \Sigma_1\text{-}\textsf{Sound}(T)$ which  implies
$T\not\vdash\Sigma_1\text{-}\textsf{Sound}(T)$.
\end{proof}

\section{Optimality of the G\"odel's second incompleteness theorem}
In this section, we construct, for any $n>0$, a $\Sigma_{n+1}$-definable and $\Sigma_{n-1}$-sound theory $\mathfrak{T}$  such that $\mathfrak{T}\vdash\Sigma_{n-1}\textsf{-Sound\,}(\mathfrak{T})$. Fix a natural number ${n}>0$ throughout this section unless otherwise noted.
The formula $\textsf{Seq}(m)$ says that $m$ is the code of a sequence of formulas, and  the length of this sequence is denoted by $\ell(m)$, and for any number $l\!<\!\ell(m)$ the $l^{\rm th}$ member of $m$ is denoted by $[m]_l$. A sequence $m$ is thus $\langle[m]_0,[m]_1,\cdots,[m]_{\ell(m)-1}\rangle$.

\begin{definition}
Fix an enumeration $\chi_0,\chi_1,
    \chi_2,\cdots$ of all the formulas such that (by the convention)  $\chi_0=\textsf{Con}(T_0)$, where $T_0=\textsf{PA} \cup\Pi_{{n}-1}\textrm{-Th}(\mathbb{N})$. We construct $\mathfrak{T}$ by recursions.
\begin{eqnarray*}
\begin{array}{rcl}
T_{0}&=&\textsf{PA}\cup\Pi_{{n}-1}\textrm{-Th}(\mathbb{N});\\
T_{i+1}&=&
\begin{cases}
T_{i}+\chi_{i}       &\qquad\qquad \text{if}~T_i\text{ is consistent with }\chi_i,\\
T_{i}+\neg\chi_i     &\qquad\qquad \textrm{otherwise;}
\end{cases}
\\ \mathfrak{T}&=&\bigcup_{i\in\mathbb{N}}T_{i}.
\end{array}
\end{eqnarray*}
\end{definition}
We will show that $\mathfrak{T}$ is the desired theory in four steps.
\begin{lemma}\label{CS023}
Let $\mathfrak{T}$ be defined as above, then
\begin{enumerate}
  \item $\mathfrak{T}$ is consistent and $\Sigma_{n-1}$-sound;
  \item $\mathfrak{T}$ is $\Sigma_{n+1}$-definable.
\end{enumerate}
\end{lemma}
\begin{proof}
(1) is trivial, and so we just prove (2). Let $\textsf{Con}_{T_0}(x)$ be defined as
$$\forall s,t,u\big[\textsf{ConjAx}_\textsf{PA} (s)\wedge\Pi_{{n}-1}\text{-}\textsf{True}(t)\rightarrow\neg\textsf{Proof\,}(u,\ulcorner s\wedge t\wedge x \rightarrow\bot\urcorner)\big],$$
and put $\textsf{Compl\,}(y)$, meaning that $y$ is a (partial) completion of $T_0$, be the formula
$$\textsf{Seq}(y)\wedge\forall j\!<\!\ell(y)\Big[\big[
\textsf{Con}_{T_0}\big(\chi_j\wedge
\bigwedge_{i< j}(y)_i
\big)\wedge[y]_j
{=}\ulcorner\chi_j\urcorner\big]
 \vee\big[
\neg\textsf{Con}_{T_0}(\chi_j\wedge\bigwedge_{i< j}(y)_i
)\wedge[y]_j{=}\ulcorner\neg
\chi_j\urcorner
\big]\Big].$$
Then the theory $\mathfrak{T}$ is definable by the following $\Sigma_{{n}+1}$-formula
$$\textsf{Axiom}_{\mathfrak{T}}(x)=_{df}\textsf{Axiom}_{T_0}(x)\vee\exists y\big(\textsf{Compl\,}(y)\wedge x{=}[y]_{\ell(y)-1}\big),$$
where $\textsf{Axiom}_{T_0}(x)=_{df}\textsf{Axiom}_\textsf{PA} (x)\vee\Pi_{{n}-1}\text{-}\textsf{True}(x)$.
\end{proof}

\begin{lemma}\label{CS013}
$\textsf{PA} +\Sigma_{{n}-1}\text{-}\textsf{Sound\,}(\textsf{PA} )\vdash\forall z\exists ! y \big(\textsf{Compl\,}(y)\wedge [y]_{\ell(y)-1}\!\in\!\{\ulcorner\!
\chi_z\!\urcorner,
\ulcorner\!\neg\chi_z\!\urcorner\}\big)$.
\end{lemma}
\begin{proof}[Proof]
Reason inside $\textsf{PA} +\Sigma_{{n}-1}\text{-}\textsf{Sound\,}(\textsf{PA} )=\textsf{PA} +\textsf{Con}(T_0)$. The existence of $y$ will be proved by induction on $z$.
\begin{itemize}
  \item For $z=0$, put $y=\langle\chi_0\rangle$ if $\textsf{Con}_{T_0}(\ulcorner\chi_0\urcorner)$ and  $y=\langle\neg\chi_0\rangle$ if $\neg\textsf{Con}_{T_0}(\ulcorner\chi_0\urcorner)$; note that by Lemma~\ref{CS005} we have $\textsf{Con}_{T_0}(\ulcorner\neg\chi_0\urcorner)$ in the latter case.
  \item Now, if   $\textsf{Compl\,}(y)\!\wedge\![y]_{\ell(y)-1}
\!\in\!\{\ulcorner\!\chi_z\!\urcorner,
\ulcorner\!\neg\chi_z\!\urcorner\}$  then
put $y'=y\hat{~}\langle\chi_{z+1}\rangle$ if $\textsf{Con}_{T_0+S}(\ulcorner\chi_{z+1}\urcorner)$ and
$y'=y\hat{~}\langle\neg\chi_{z+1}\rangle$ if $\neg\textsf{Con}_{T_0+S}(\ulcorner\chi_{z+1}\urcorner)$, where $S$ is the set  $\{[y]_0,\cdots,[y]_z\}$ and{~ }$\hat{~}${~ }denotes the concatenation operation. Note again that by Lemma~\ref{CS005}, $\textsf{Con}_{T_0+S}(\ulcorner\neg\chi_{z+1}\urcorner)$ in the latter case. It can then be easily seen that 
$\textsf{Compl\,}(y')\wedge[y']_{\ell(y')-1}\!\in\!
\{\ulcorner\!\chi_{z+1}\!\urcorner,
\ulcorner\!\neg\chi_{z+1}\!\urcorner\}$.
\end{itemize}
The uniqueness of $y$ will again be proved by induction on $z$.
\begin{itemize}
  \item For $z=0$, if for some $y$ and $y'$ we have
  $\textsf{Compl\,}(y)\wedge[y]_{\ell(y)-1}
\!\in\!\{\ulcorner\!\chi_0\!\urcorner,
\ulcorner\!\neg\chi_0\!\urcorner\}$ and $\textsf{Compl\,}(y')\wedge[y']_{\ell(y')-1}
\!\in\!\{\ulcorner\!\chi_0\!\urcorner,
\ulcorner\!\neg\chi_0\!\urcorner\}$
  then if $y\neq y'$ we should have either  $[y]_0{=}\ulcorner\chi_0\urcorner$, $[y']_0{=}\ulcorner\neg\chi_0\urcorner$ or   $[y]_0{=}\ulcorner\neg\chi_0\urcorner$, $[y']_0{=}\ulcorner\chi_0\urcorner$. Then  we must have $\textsf{Con}_{T_0}(\ulcorner\chi_0\urcorner)\wedge\neg\textsf{Con}_{T_0}(\ulcorner\chi_0\urcorner)$ in both cases; contradictions.
\item For $z+1$ assume that both $\textsf{Compl\,}(y) \wedge [y]_{\ell(y)-1}
\!\in\!\{\ulcorner\!\chi_{z+1}\!\urcorner,
\ulcorner\!\neg\chi_{z+1}\!\urcorner\}$ and $\textsf{Compl\,}(y') \wedge [y']_{\ell(y')-1}
\!\in\!\{\ulcorner\!\chi_{z+1}\!\urcorner,
\ulcorner\!\neg\chi_{z+1}\!\urcorner\}$ hold. Then, if
for a sequence $s$ we denote $\langle [s]_0,\cdots,[s]_{j-1}\rangle$ by $s\downharpoonright j$,
 $\textsf{Compl\,}\big(y\!\downharpoonright\!(z+1)\big)\wedge\textsf{Compl\,}\big(y'\!\downharpoonright\!(z+1)\big)$ holds, and so by the induction hypothesis   $y\!\downharpoonright\!(z+1){=}y'\!\downharpoonright\!(z+1)$.
 Thus it remains to show that $[y]_{z+1}{=}[y']_{z+1}$. If $[y]_{z+1}{\neq}[y']_{z+1}$ then
either $[y]_{z+1}{=}\ulcorner\chi_{z+1}\urcorner$, $[y']_{z+1}{=}\ulcorner\neg\chi_{z+1}\urcorner$  or $[y]_{z+1}{=}\ulcorner\neg\chi_{z+1}\urcorner$, $[y']_{z+1}{=}\ulcorner\chi_{z+1}\urcorner$, and then, just like before,
  $\textsf{Con}_{T_0}\big(\ulcorner\chi_{z+1}\wedge
\bigwedge_{i\leqslant z}[y]_i\urcorner\big)\wedge\neg\textsf{Con}_{T_0}\big(\ulcorner\chi_{z+1}\wedge
\bigwedge_{i\leqslant z}[y']_i\big\urcorner)$
should hold in both cases;
contradiction with $y\!\downharpoonright\!(z+1){=}y'\!\downharpoonright\!(z+1)$. \qedhere
\end{itemize}
\end{proof}

\begin{lemma}\label{CS012}
$\textsf{PA} +\Sigma_{{n}-1}\text{-}\textsf{Sound\,}(\textsf{PA} )\vdash\Sigma_{{n}-1}\text{-}\textsf{Sound\,}(\mathfrak{T})$.
\end{lemma}
\begin{proof}[Proof]
Reason inside $\textsf{PA} +\Sigma_{{n}-1}\text{-}\textsf{Sound\,}(\textsf{PA} )=\textsf{PA} +\textsf{Con}(T_0)$. Take $\langle\psi_0,\cdots,\psi_l\rangle$ to be any sequence of the axioms of $T_0$ and $\langle\vartheta_0,\cdots,\vartheta_k\rangle$ to  be any sequence of formulas for which there are $\langle y_0,\cdots,y_k\rangle$ such that $\bigwedge_{i\leqslant k}\big[\textsf{Compl\,}(y_i) \wedge [y_i]_{\ell(y_i)-1}
{=}\,\ulcorner\!\vartheta_i\!\urcorner\big]$. By Lemma~\ref{CS013} all $y_i$'s are in initial segments of $u=\max\{y_0,\cdots,y_k\}$. So, all  $\vartheta_i$'s appear in the list  $[u]_0,\cdots,[u]_{\ell(u)-1}$. It follows from $\textsf{Compl\,}(u)$ that $\textsf{Con}_{T_0}(\ulcorner \bigwedge_{i<\ell(u)}[u]_i\urcorner)$, hence  we have $\textsf{Con}_{T_0}(\ulcorner\bigwedge_{i\leqslant k}\vartheta_i\urcorner)$, so  $\textsf{Proof\,}(u,\ulcorner \bigwedge_{i\leqslant k}\vartheta_i\wedge \bigwedge_{j\leqslant l}\psi_j\rightarrow\bot\urcorner)$ can hold for no $u$. Now, since any sequence of the axioms of $\mathfrak{T}$ can be rearranged as $\langle\psi_0,\cdots,\psi_l,\vartheta_0,\cdots,\vartheta_k\rangle$ where $\psi_j$'s and $\vartheta_i$'s are as above, $\textsf{Con}(\mathfrak{T})$ holds.

\noindent Therefore, $\textsf{PA} +\Sigma_{{n}-1}\text{-}\textsf{Sound\,}(\textsf{PA} )\vdash\textsf{Con}(\mathfrak{T})$, and then our conclusion follows from the fact that $\Sigma_{{n}-1}\text{-}\textsf{Sound\,}(\mathfrak{T})=_{df}\textsf{Con}(\mathfrak{T}+\Pi_{n-1}\textrm{-Th}(\mathbb{N}))=\textsf{Con}(\mathfrak{T})$ since $\mathfrak{T}+\Pi_{n-1}\textrm{-Th}(\mathbb{N})=\mathfrak{T}$.
\end{proof}

\begin{theorem}\label{CS015}
For any $n  \geq 1$, there exists a $\Delta_{n +1}$-definable and $\Sigma_{n -1}$-sound theory $\mathfrak{T}$ which proves self $\Sigma_{n -1}$-soundness:
$\mathfrak{T}\vdash\Sigma_{n -1}\text{-}\textsf{Sound\,}(\mathfrak{T})$.
\end{theorem}
\begin{proof}[Proof]
The theory $\mathfrak{T}$ constructed above is $\Sigma_{n +1}$-definable, and since it is complete, it must be $\Pi_{n +1}$-definable as well. To see it more directly, note that for all $j\in \mathbb{N}$
$$ \chi_j \in {\mathfrak{T}} \Longleftrightarrow  \mathbb{N}\vDash\textsf{Axiom}_{T_0}(\ulcorner \chi_j \urcorner) \vee \forall y(\textsf{Compl\,}(y) \wedge j\!<\!\ell(y)\rightarrow \ulcorner \chi_j \urcorner {=} [y]_{j}).$$
 Since $\chi_0=\textsf{Con}(T_0)$ is consistent with $T_0$ (i.e. $\mathbb{N}\vDash \textsf{Con}_{T_0}(\ulcorner \chi_0\urcorner)$), then $\chi_0=\textsf{Con}(T_0)\in T_1$, and so $\mathfrak{T}\vdash\textsf{Con}(T_0)$.
 Therefore, noting that $\Sigma_{n-1}\text{-}\textsf{Sound\,}(\textsf{PA})=_{df}\textsf{Con}(\textsf{PA}+\Pi_{n-1}\textrm{-Th}(\mathbb{N}))=\textsf{Con}(T_0)$ and $\textsf{PA}\subseteq\mathfrak{T}$, Lemma~\ref{CS012} implies that  $\mathfrak{T}\vdash\Sigma_{n -1}\text{-}\textsf{Sound\,}(\mathfrak{T})$.
\end{proof}

\section{Concluding Remarks}
A special case of G\"odel's second incompleteness theorem for $\Sigma_n$-soundness of ${\sf PA}$ follows from the well-known facts on strong provability predicates and their modal logics (see e.g. \cite{Rppa2005,Ospa1993}) and it could be extended to $\Sigma_n$-definable and explicit (provable) extensions of  ${\sf PA}$. So, no $\Sigma_{n}$-definable, $\Sigma_{n-1}$-sound and explicit extension of ${\sf PA}$ can prove its own $\Sigma_{n-1}$-soundness (Theorem~\ref{CS007}---which generalizes Theorem~6 of \cite{Eldt1975}). We strengthened this result by deleting the requirement of ``explicit extension of $\textsf{PA}$'' (Theorem \ref{CS008}).
 The  optimality of this result, in a sense, follows from the fact that a complete $\Delta_{n+1}$-definable and $\Sigma_{n-1}$-sound theory (which is an explicit extension of ${\sf PA}$) may prove its own $\Sigma_{n-1}$-soundness (Theorem~\ref{CS015}---which generalizes an example  of \cite{Eldt1975} reconstructed in \cite{Elmc2012}).

 \section*{Acknowledgements}
 The first author would like to thank  his supervisor Professor Zhuanghu Liu at Peking University for creating a free environment for study and research, and thanks to Professor Yue Yang at National University of Singapore for leading him to the filed of incompleteness, and also special thanks to Professor Xianghui Shi at Beijing Normal University for teaching him a lot of mathematical logic. This is a part of the Ph.D. thesis of the second author under the supervision of Professor  Saeed Salehi to whom he is most grateful for valuable suggestions and completely rewriting the paper. The author also thanks Professor Ali Enayat (Gothenburg) for completely editing  the paper and fruitful suggestions.

\bibliographystyle{plain}

\begin{thebibliography}{20}
\bibitem{Rppa2005}
L.~D. Beklemishev.
\newblock {R}eflection principles and provability algebras in formal
  arithmetic.
\newblock {\em Russian Mathematical Surveys}, 60(2):197, 2005.

\bibitem{Tpol2003}
G.~Boolos.
\newblock {\em The Logic of Provability}.
\newblock Cambridge University Press, 2003.

\bibitem{Mfoa1998}
P.~H{\'a}jek and P.~Pudl{\'a}k.
\newblock {\em {M}etamathematics of First-Order Arithmetic}.
\newblock 1998.

\bibitem{Ospa1993}
K.~N. Ignatiev.
\newblock {O}n strong provability predicates and the associated modal
  logics.
\newblock {\em The Journal of Symbolic Logic}, 58(01):249--290, 1993.

\bibitem{Nsgt2011}
D.~Isaacson.
\newblock {N}ecessary and sufficient conditionsfor undecidability of the
  {G}\"{o}del sentence and its truth.
\newblock In D.~DeVidi, M.~Hallett, and P.~Clarke, editors, {\em Logic,
  Mathematics, Philosophy: Vintage Enthusiasms. Essays in honour of John L.
  Bell}, volume~75 of {\em The Western Ontario Series in Philosophy of
  Science}, pages 135--152. Springer Netherlands, 2011.

\bibitem{Eldt1975}
R.~G. Jeroslow.
\newblock {E}xperimental logics and {$\Delta^0_2$}-theories.
\newblock {\em Journal of Philosophical Logic}, 4(4):253--267, 1975.

\bibitem{Elmc2012}
M.~Kas{\aa}.
\newblock {E}xperimental logics, mechanism and knowable consistency.
\newblock {\em Theoria}, 78(3):213--224, 2012.

\bibitem{Mopa1991}
R.~Kaye.
\newblock {\em Models of Peano Arithmetic}.
\newblock Oxford Science Publications, 1991.

\bibitem{Acml2006}
W.~Rautenberg.
\newblock {\em A Concise Introduction to Mathematical Logic}.
\newblock Springer, 2006.

\bibitem{Aigt2013}
P.~Smith.
\newblock {\em An Introduction to G\"{o}del's Theorems}.
\newblock Cambridge Introductions to Philosophy. Cambridge University Press,
  2nd edition, 2013.





\end{thebibliography}

\end{document}